\documentclass[a4paper,reqno,12pt,draft]{amsart}
\usepackage{amsmath}
\usepackage{amssymb}
\usepackage{amsfonts}
\usepackage{amsthm}
\usepackage{enumerate}
\usepackage{comment}
\usepackage{cite}
\usepackage[dvips]{graphicx,color}
\usepackage{delarray}
\usepackage{ascmac}
\usepackage{fancybox}
\usepackage{lineno}
\usepackage[hidelinks,draft=false]{hyperref}
\theoremstyle{plain}
\newtheorem{theorem}{Theorem}[section]
\newtheorem{corollary}[theorem]{Corollary}

\newtheorem{proposition}[theorem]{Proposition}
\newtheorem{lemma}[theorem]{Lemma}
\newtheorem{remark}[theorem]{Remark}

\numberwithin{equation}{section}

\numberwithin{equation}{section}

\def\XXint#1#2#3{{\setbox0=\hbox{$#1{#2#3}{\int}$}
\vcenter{\hbox{$#2#3$}}\kern-.5\wd0}}

\begin{document}
\title[Bifurcation diagrams with singular weights in two dimensions]{Bifurcation diagrams for semilinear elliptic equations with singular weights in two dimensions}
\author{Kenta Kumagai}
\address{Department of Mathematics, Tokyo Institute of Technology}
\thanks{This work was supported by JSPS KAKENHI Grant Number 23KJ0949}
\email{kumagai.k.ah@m.titech.ac.jp}
\date{\today}

\begin{abstract}
We consider the bifurcation diagram of radial solutions for the Gelfand problem with a positive radially symmetric weight in the unit ball. We deal with the exponential nonlinearity and a power-type nonlinearity. When the weight is constant, it is well-known that the bifurcation curve exhibits three different types depending on the dimension and the exponent of power for higher dimensions, while the curve exhibits only one type in two dimensions. 

In this paper, we succeed in realizing in two dimensions a phenomenon such that the bifurcation curve exhibits all of the three types, by choosing the weight appropriately. In particular, to the best of the author's knowledge, it is the first result to establish in two dimensions the bifurcation curve having no turning points.
\end{abstract}
\keywords{Semilinear elliptic equation, Bifurcation diagram, Two dimensions, Turning points, Stability}
    \subjclass[2020]{Primary 35J61, 35B32; Secondary 35J25, 35B35}

\maketitle

\section{Introduction}
Let $N=2$ and $B_1\subset \mathbb{R}^N$ be the unit ball. We consider the bifurcation diagram of radial solutions for the semilinear elliptic problem
\begin{equation}
\label{gelfand}
-\Delta u=\lambda V_k(|x|)f(u) \hspace{2mm} \text{in } B_1, \hspace{4mm}
u>0 \hspace{2mm}\text{in } B_1, \hspace{4mm}
u=0  \hspace{2mm}\text{on } \partial B_1,
\end{equation}
where $\lambda>0$ is a parameter and
\begin{equation}
\label{pote}
    V_k(r):=\frac{1}{r^2 (-\log (r/e))^{2+k}} \hspace{4mm} \text{with $k>0$}.
\end{equation}
In this paper, we deal with the following two types of the nonlinearities $f(u)=e^u$ and $f(u)=(1+u)^p$ with $p>k+1$.
\subsection{Classical case: \texorpdfstring{$V_k=1$}{LG}}
\label{a1}
In this case, by the symmetric result of Gidas, Ni, and Nirenberg \cite{Gidas}, every solution of \eqref{gelfand} is radially symmetric. Moreover, it is known \cite{korman, korman2014, Mi2015} that the set of solutions of \eqref{gelfand} is an unbounded analytic curve emanating from $(0,0)$ and described as $\{(\lambda(\alpha), u(r, \alpha)); \alpha>0\}$, where $u(r,\alpha)$ is the solution satisfying $\lVert u \rVert_{L^{\infty}(B_1)}=u(0)=\alpha$. We call this set $\{(\lambda(\alpha),\alpha); \alpha>0\}$ the bifurcation curve. A celebrated result of Joseph and Lundgren \cite{JL} states that the bifurcation curve exhibits the following three different types depending on the dimension $N$ and the exponent of power.
\begin{itemize}
    \item[{($0$)}] The curve emanating from $(0,0)$ goes to $\lambda=\lambda^{*}$ with some $\lambda^{*}>0$, bends back at $\lambda^{*}$, and then converges to $\lambda=0$ monotonically as $\alpha\to\infty$.
    \item[{(I)}] The curve emanating from $(0,0)$ turns infinitely many times around $\lambda=\lambda_{*}$ with some $\lambda_{*}>0$. In addition, the curve converges to $\lambda=\lambda_{*}$ as $\alpha\to\infty$.
    \item[{(II)}] The curve emanating from $(0,0)$ monotonically converges to $\lambda=\lambda_{*}$ with some $\lambda_{*}>0$ as $\alpha\to\infty$.
\end{itemize}
More precisely, they proved by using the Emden-Fowler transformation and a phase plane analysis that when $f(u)=e^u$,
the bifurcation diagram is of Type $0$ if $N=2$, of Type I if $3\le N\le 9$, and of Type II if $10\le N$. On the other hand, they proved that when $f(u)=(1+u)^q$, 
the bifurcation diagram is of Type 0 if $1<q\le q_c$, of Type I if $q_c<q<q^{+}_{\mathrm{JL}}$, and of Type II if $q^{+}_{\mathrm{JL}}\le q<\infty$, where 
\begin{equation*}
q_{c}=\begin{cases}
   \frac{N+2}{N-2} &\text{if $N\ge 3$,}\\
    \infty&\text{if $N=2$,}
\end{cases}
\hspace{4mm}\text{and} \hspace{4mm}q^{+}_{\mathrm{JL}}=
\begin{cases}
    1+\frac{4}{N-4-2\sqrt{N-1}}&\text{if $N\ge 11$,}\\
    \infty&\text{if $N\le  10$.}
\end{cases}
\end{equation*}
Here, we say that $(\lambda, U)$ is a radial singular solution of \eqref{gelfand} if $U\in C^2(0,1]$ satisfies \eqref{gelfand} and $U(r)\to \infty$ as $r\to 0$. Then,
for the case $N\ge 3$, they proved that there exists a singular solution 
\begin{equation*}
    (\lambda_{*},U_{*})=
    \begin{cases}
         (2(N-2),-2\log|x|) &\text{if $f(u)=e^u$,}\\
    (\theta(N-2-\theta),|x|^{-\theta}-1)&\text{if $f(u)=(1+u)^q$ with $q>q_s$}
    \end{cases}
\end{equation*}
with $q_s:=\frac{N}{N-2}$ and $\theta:=\frac{2}{q-1}$ such that the bifurcation curve converges to $(\lambda_{*}, U_{*})$ when $f(u)=e^u$ with $N\ge 3$ or $f(u)=(1+u)^q$ with $q>q_c$. Moreover, Brezis and V\'azquez \cite{Br} studied the stability of $U_{*}$ by using the Hardy inequality. Here, we mean that $U_{*}$ is stable if for any $\xi\in C^{0,1}_{0}(B_1)$, it follows that
\begin{align}
\label{sta}
Q_{U_{*}}(\xi):=\int_{B_1}|\nabla\xi|^2\,dx-\int_{B_1}\lambda V_k(|x|) f'(U_{*})\xi^2\,dx\ge 0.
\end{align}
As a result, they proved that when $f(u)=e^u$, $U_{*}\in H^1(B_1)$ is always satisfied. Moreover, $U_{*}$ is stable if and only if $N\ge 10$. On the other hand, when $f(u)=(1+u)^q$ with $q>q_s$, $U_{*}\in H^1(B_1)$ is satisfied if and only if $q_c<q$. Moreover, $U_{*}$ is stable if and only if $q\le q^{-}_{\mathrm{JL}}$ or $q\ge q^{+}_{\mathrm{JL}}$, where $q^{-}_{\mathrm{JL}}=1+\frac{4}{N-4+2\sqrt{N-1}}$. Here, we remark that $1<q_{s}<q^{-}_{\mathrm{JL}}<q_{c}<q^{+}_{\mathrm{JL}}$. In addition, they showed in \cite{Br} the following important relation between the stability of singular solutions and the bifurcation structure for all non-negative non-decreasing and convex nonlinearities $f$: if a stable singular solution $U\in H^1(B_1)$ of \eqref{gelfand} exists with some $\lambda$, then the bifurcation diagram is of Type II. From this result and the stability analysis stated above, we can also confirm that the bifurcation diagram is of Type II if $f(u)=e^u$ with $N\ge 10$ or $f(u)=(1+u)^q$ with $q\ge q^{+}_{\mathrm{JL}}$. Then, there have been many studies trying to study the bifurcation structure and some properties of singular solutions for various nonlinearities $f$. See \cite{Nor,chend,chen,Flore,Marius,Guowei,Kiwei,Luo,Lin,Merle,Mi2014,Mi2015,Mi2018,Mi2020,Mi2023}.

 

In contrast to the case $N\ge 3$, the bifurcation curve exhibits only Type 0 in two dimensions when $f(u)=e^u$ or $f(u)=(1+u)^q$. Moreover, until recently, 
the nonlinearitiy $f$ for which the bifurcation curve exhibits Type I had not been confirmed when $N=2$. Recently, Naimen \cite{Naimen} proved the oscillation of the bifurcation curve for a class of nonlinearities including $f(u)=e^{u^p}$ with $p>2$ when $N=2$. Moreover, the author \cite{kuma2024} proved that the curve has infinitely many turning points for general supercritical nonlinearities in the sense of the Trudinger-Morser imbedding when $N=2$. We also remark that the oscillation phenomenon is confirmed in \cite{Bou} for a cleverly set problem in two dimensions. However, it is not guaranteed in \cite{Bou,kuma2024,Naimen} that $\lambda(\alpha)\to \lambda_{*}$ for some $\lambda_{*}>0$.
Moreover, it is known by \cite{cabrecapella} that the bifurcation curve does not exhibit Type II for any non-negative non-decreasing nonlinearities in two dimensions in the unweighted case. 

Our motivation is to realize in two dimensions a phenomenon such that the bifurcation curve exhibits all of the three types obtained in \cite{JL}, by choosing the weight $V_k$ appropriately. 
\subsection{Weighted case}
Motivated by the classical case, we try to consider the bifurcation diagram of radial solutions for \eqref{gelfand} with the weight $V_k$ satisfying \eqref{pote}.
Here, we say that a pair $(\lambda, u)$ is a radial solution of \eqref{gelfand} if $u\in C_{\mathrm{rad}}^0[0,1]\cap C^2(0,1]$ and $u$ satisfies \eqref{gelfand}.
We note that the condition $k>0$ is natural: there exists no radial solution if $k\le 0$ (see Lemma \ref{nlem}). For the weighted case in two dimensions, we find the following exponents corresponding to $q_{s}$, $q_{c}$, and $q^{\pm}_{\mathrm{JL}}$ stated above:
\begin{equation*}
    p_s=k+1,\hspace{0.5mm} p^{-}_{\mathrm{JL}}= \frac{2k}{1-k+\sqrt{k(k+2)}},\hspace{0.5mm} p_c=2k+1,\hspace{0.5mm}p^{+}_{\mathrm{JL}}=
  \begin{cases}
    \frac{2k}{1-k-\sqrt{k(k+2)}} & \text{\hspace{-1.5mm}$k< \frac{1}{4}$}, \\
     \infty  & \text{\hspace{-1.5mm}$k\ge \frac{1}{4}$}.
  \end{cases}
\end{equation*}
Here, we remark that $1<p_s<p^{-}_{\mathrm{JL}}<p_c<p^{+}_{\mathrm{JL}}$.

We first obtain the following theorems by
using the specific changes of variables used in \cite{korman2014,Kuma} for $f(u)=e^u$ and used
in \cite{korman2018} for $f(u)=(1+u)^p$. 
\begin{theorem}
\label{diagramth}
Let $N=2$, $k>0$, $f(u)=e^u$ and $V_k$ be that in \eqref{pote}. Then, the set of radial solutions of \eqref{gelfand} is described as
\begin{equation*}
\{\left(\lambda(\beta), u(r, \alpha(\beta))\right); \beta\in\mathbb{R}\}\hspace{4mm}\text{with}\hspace{4mm}\alpha(\beta)=\beta-\log \lambda(\beta),
\end{equation*}
where $\alpha(\beta):=\lVert u \rVert_{L^{\infty}(B_1)}=u(0)$. Every radial solution $u\in H^1_{0}(B_1)$ and $u$ satisfies \eqref{gelfand} in the weak sense. 
Moreover, $\mathcal{C}:=\{(\lambda(\beta), \alpha(\beta)); \beta\in \mathbb{R}\}$ is an unbounded analytic curve emanating from $(0,0)$ and there exists $0<\lambda^{*}<\infty$ such that $\lambda(\beta)\le \lambda^{*}$ for all $\beta$. 
\end{theorem}
\begin{theorem}
\label{diagramthp}
Let $N=2$, $k>0$ and $V_k$ be that in \eqref{pote}. We assume that $f(u)=(1+u)^p$ with $p>p_s$. Then, there exists $\beta^{*}\in (0,\infty]$ depending only on $p$ such that 
the set of radial solutions of \eqref{gelfand} is described as
\begin{equation*}
\{\left(\lambda(\beta), u(r, \alpha(\beta))\right); \beta\in(0,\beta^{*})\}\hspace{4mm}\text{with}\hspace{4mm}\alpha(\beta)=\lambda(\beta)^{-\frac{1}{p-1}}\beta-1,
\end{equation*}
where $\alpha(\beta):=\lVert u \rVert_{L^{\infty}(B_1)}=u(0)$. Every radial solution $u\in H^1_{0}(B_1)$ and $u$ satisfies \eqref{gelfand} in the weak sense.
Moreover, $\mathcal{C}:=\{(\lambda(\beta), \alpha(\beta)); \beta\in (0,\beta^{*})\}$ is an unbounded analytic curve emanating from $(0,0)$  and there exists $0<\lambda^{*}<\infty$ such that $\lambda(\beta)\le \lambda^{*}$ for all $\beta$. In addition, $\beta^{*}=\infty$ if $p\ge p_c$ and $\beta^{*}<\infty$ if $p_s<p<p_c$.
\end{theorem}
We call the curve the bifurcation curve and we say that $(\lambda(\beta),\alpha(\beta))$ is a turning point if $\lambda(\beta)$ is an extreme point. Thanks to the above theorems, we verify that  $(\lambda(\beta_0),\alpha(\beta_0))$ is a turning point if and only if $\lambda$ is represented as a graph of $\alpha$ in some neighborhood of $(\lambda(\beta),\alpha(\beta))$ and $\frac{d\lambda}{d\alpha}(\alpha(\beta_0))=0$. On the contrary, it is not guaranteed whether $\lambda$ is globally parameterized by $\alpha$.

For the special weight $V_k$, we can get a singular solution explicitly. Moreover, we find new Emden-Fowler type transformations for the specific weights $V_k$. It enables us to use a phase plane analysis, and thus we get the following
\begin{theorem}
\label{singularth}
Let $N=2$, $k>0$ and $V_k$ be that in \eqref{pote}. Then,
\begin{itemize}
    \item[{\rm{(i)}}] when $f(u)=e^u$,  \eqref{gelfand} has a unique radial singular solution
    \begin{equation*}
        (\lambda_*, U_*):=\left(k,\log\left(\log(r/e)\right)\right).
    \end{equation*}
Moreover, $U_{*}\in H^1_{0}(B_1)$ and we have  
    \begin{equation*}
\lambda(\beta)\to \lambda_{*} \hspace{4mm}\text{and}\hspace{4mm} u(r,\beta)\to U_* (r) \hspace{4mm} \text{in $C^2_{\mathrm{loc}}(0,1]$} \hspace{4mm} \text{as $\beta\to\infty$}. 
 \end{equation*}
    \item[\rm{{(ii)}}] when $f(u)=(1+u)^p$ with $p>p_s$, \eqref{gelfand} has a radial singular solution 
    \begin{equation*}
    (\lambda_*, U_*):=\left(\frac{k}{p-1}(1-\frac{k}{p-1}), \left(\log(r/e) \right)^{\frac{k}{p-1}}-1\right).
    \end{equation*}
 Moreover, $U_{*}\in H^1_{0}(B_1)$ if and only if $p>p_c$. In addition, 
 when $p>p_c$, the radial singular solution of \eqref{gelfand} is unique and we have 
 \begin{equation*}
\lambda(\beta)\to \lambda_{*} \hspace{4mm}\text{and}\hspace{4mm} u(r,\beta)\to U_* (r) \hspace{4mm} \text{in $C^2_{\mathrm{loc}}(0,1]$} \hspace{4mm} \text{as $\beta\to\infty$.} 
 \end{equation*}
\end{itemize}
\end{theorem}
Next, we define the Morse index $m(U_{*})$ as the maximal dimension of a subspace $X\subset H^1(B_1)$ such that $Q_{U_{*}}(\xi)<0$ for all $\xi\in X\setminus\{0\}$, where $Q_{U_{*}}$ is that in \eqref{sta}. We remark that $U_{*}$ is stable if and only if $m(U_{*})=0$. In addition, we remark that the Morse index of the singular solution plays a key role in the bifurcation structure in the classical case (see \cite{Br,Dan,Guowei,Mi2014}). In the next theorem, we study the Morse index of the singular solutions.
\begin{theorem}
\label{morseth}
Let $N=2$, $k>0$ and $V_k$ be that in \eqref{pote}. We assume that $(\lambda_{*},U_{*})$ be that in Theorem \ref{singularth}. Then, 
\begin{itemize}
    \item[\rm{{(i)}}] when $f(u)=e^u$, we have $m(U_*)=0$ for $k\le \frac{1}{4}$ and $m(U_{*})=\infty$ for $k>\frac{1}{4}$.
\item[\rm{{(ii)}}] when $f(u)=(1+u)^p$ with $p>p_s$, we have $m(U_*)=0$ for $p\le p^{-}_{\mathrm{JL}}$ or $p\ge p^{+}_{\mathrm{JL}}$, and $m(U_{*})=\infty$ for $p^{-}_{\mathrm{JL}}<p<p^{+}_{\mathrm{JL}}$.
\end{itemize}
\end{theorem}
As mentioned in subsection \ref{a1}, the stability/instability of singular solutions is obtained by the Hardy inequality in the classical case with $N\ge 3$. On the contrary, we mention that in the weighted case, the exponents $\frac{1}{4}$ and $p^{\pm}_{\mathrm{JL}}$ arise from the best constant of the critical Hardy inequality and this inequality plays a key role in the stability/instability of the singular solution. Thanks to the stability analysis and the phase plane analysis,
we get the main theorems.
\begin{theorem}
\label{mainthe}
Let $N=2$ and $f(u)=e^u$ and $V_k$ is that in \eqref{pote}.
Then, the bifurcation diagram of \eqref{gelfand} is of 
\begin{itemize}
    \item[\rm{{(i)}}] Type I if $k>\frac{1}{4}$.
    \item[\rm{{(ii)}}] Type II if $k\le \frac{1}{4}$. 
\end{itemize}
Moreover, $\lambda$ is globally parameterized by $\alpha$ if $k\le \frac{1}{4}$.
\end{theorem}
\begin{theorem}
\label{mainthp}
Let $N=2$, $k>0$, $f(u)=(1+u)^p$ and $V_k$ be that in \eqref{pote}.
Then, the bifurcation diagram of \eqref{gelfand} is of 
\begin{itemize}
    \item[\rm{{(i)}}] Type $0$ if 
    $p_s<p\le p_c$.
    \item[\rm{{(ii)}}] Type I if $p_c < p<p^{+}_{\mathrm{JL}}$.
    \item[\rm{{(iii)}}] Type II if $p\ge p^{+}_{\mathrm{JL}}$.
\end{itemize}
Moreover, $\lambda$ is globally parameterized by $\alpha$ if $p_s<p\le p_c$ or $p\ge p_{\mathrm{JL}}$.
\end{theorem}
Here, we explain related works for the weighted case. When $N\ge 3$, the bifurcation structure changes depending on the singularity of the weight at $r=0$. Indeed, when $f(u)=e^u$, Korman \cite{korman2014} proved for the weight $\hat{V}_k(r)=r^k$ with $k>0$ that the bifurcation diagram for radial solutions is of Type I if $3\le N<10+4k$. Moreover, Bae \cite{Bae} proved for the same weight with $k>-2$ that the bifurcation diagram is of Type II if $10+4k\le N$. On the other hand, when $f(u)=(1+u)^q$, for the same weight with $k>-2$, it is known by \cite{korman2018,kor2020} that the bifurcation diagram for radial solutions is of Type 0 if $1<q\le \hat{q}_{c}$, of Type I if $\hat{q}_{c}<q<\hat{q}_{\mathrm{JL}}$, and of Type II if $\hat{q}_{\mathrm{JL}}\le q$,
where $\hat{q}_{c}=\frac{N+2+2k}{N-2}$ if $N\ge 3$ and $\hat{q}_c=\infty$ if $N=2$. In addition, $\hat{q}_{\mathrm{JL}}$ is defined as
\begin{equation*}
\hat{q}_{\mathrm{JL}}=
   \begin{cases}
   \frac{(N-2)^2 - 2(k+2)(N+k)+2(k+2)\sqrt{(N+k)^2-(N-2)^2}}{(N-2)(N-10-4k)} &\text{if $10+4k<N$}, \\
   \infty &\text{if $3\le N \le 10+4k$}.
    \end{cases}
\end{equation*}
For related works with general weights, see \cite{Bae,Kuma}. In addition, we also mention the results concerning the separation and intersection property for radial solutions of $-\Delta v(r,\beta) = r^k f(v(r,\beta))$ in $\mathbb{R}^N$ with $N\ge 3$ and $k>-2$ satisfying $v(0,\beta)=\beta$. Bae \cite{Bae} showed that when $f(v)=e^v$, any radial solution $v(r,\beta)$ and $v(r,\gamma)$ intersect infinitely many times for $0<\beta<\gamma$ if $3\le N < 10+4k$, while it follows that $v(r,\beta)<v(r,\gamma)$ for $0<\beta<\gamma$ if $10+4k\le N$. In addition, Ni and Yotsutani \cite{Yotsu} showed that when $f(v)=v^q$, $v(r,\beta)$ has a finite zero for every $\beta>0$ if $1<q<\hat{q}_c$ and $v(r,\beta)$ is a positive entire solution for every $\beta>0$ if $q\ge \hat{q}_c$. Then, Wang \cite{wang} showed that when $f(v)=v^q$ with $\hat{q}_{c}<q<\hat{q}_{\mathrm{JL}}$, any radial solution $v(r,\beta)$ and $v(r,\gamma)$ intersect infinitely many times for $0<\beta<\gamma$, while when $\hat{q}_{\mathrm{JL}}\le q$, it follows that $v(r,\beta)<v(r,\gamma)$ for $0<\beta<\gamma$.
For more general results on this direction, see \cite{Bae2003,Bae2004,Bae2002-2,Bae2013,Bae,Baenaito2,Baenaito,Baeni,LLD,Yotsu,Yana,Yana2}.

On the contrary, when $N=2$, to the best of the author's knowledge, any changes of bifurcation structure depending on the weights have not been confirmed. Moreover, even bifurcations which satisfy the oscillation phenomenon have not been confirmed except in \cite{Bou,kuma2024,Naimen}, and bifurcations of Type II have not been confirmed in the literature.
The novelty of this paper is not only to establish bifurcations of Type II in two dimensions for the first time, but also to realize analogues of the phenomena obtained by \cite{Br,JL} in two dimensions by considering the suitable weight $V_k$.
Moreover, we obtain the following separation and intersection property analogous to the results in \cite{Bae,Yotsu,wang}.

\begin{theorem}
    \label{sepathm}
Let $N=2$, $\beta<\gamma$, $k>0$ and $V_k$ be that in \eqref{pote}. We consider the radial solution $v=v(r,\beta)$ of 
$-\Delta v(r,\beta) = V_{k}(|x|)f(v(r,\beta))$ in $B_e$ satisfying $v(0,\beta)=\beta$.
Then, the solution $v\in C^0(0,e)\cap C^2(0,e)$ is unique. Moreover, 
\begin{itemize}
\item[{\rm{(i)}}] when $f(v)=e^v$, we obtain that $v(r, \beta)<v(r,\gamma)$ in $B_e$ if $k\le \frac{1}{4}$. Moreover, $v(r,\beta)$ and $v(r, \gamma)$ intersect infinitely many times if $k>\frac{1}{4}$.
\item [\rm{{(ii)}}] when $f(v)=|v|^p$ with $p_s<p<p_c$, then we have $v(r,\beta)=0$ with some $r<e$ for any $\beta>0$.
\item[\rm{{(iii)}}] when $f(v)=|v|^p$ with $p\ge p_c$ and $\beta>0$, we obtain that $v(r,\beta)>0$ in $B_e$. Moreover, we get
$v(r,\beta)<v(r,\gamma)$ in $B_e$ if $p\ge p^{+}_{\mathrm{JL}}$, while $v(r,\beta)$ intersects $v(r,\gamma)$ infinitely many times if $p_c<p<p^{+}_{\mathrm{JL}}$.
\end{itemize}
\end{theorem}

This paper is organized as follows. In Section \ref{expsec}, we deal with the exponential case. In Section \ref{powersec}, we deal with the case $f(u)=(1+u)^p$.

\section{The exponential case}
\label{expsec}
In this section, we deal with the exponential case $f(u)=e^u$. We begin by 
introducing a specific change of variables which is used in \cite{korman2014,Kuma}. Assume that $(\lambda,u)$ is a radial solution of \eqref{gelfand} with $\lVert u \rVert_{L^\infty(B_1)}=\alpha$. We define 
$v:=u+\log \lambda$ and $\beta:=\alpha+\log \lambda$. Then, $v$ can be extended to $(0,e)$ such that $v$ satisfies 
\begin{equation}
\label{ev}
\left\{
\begin{alignedat}{4}
&v''+  \frac{1}{r}v' + V_k (r) e^v=0, \hspace{4mm}0<r<e,\\
&v(0)=\beta, \hspace{4mm}v\in C^2(0,e)\cap C^{0}[0,e).
\end{alignedat}
\right.
\end{equation}

In the following, we consider a solution of the equation
\begin{equation}
\label{apriorieq}
v''+  \frac{1}{r}v' + V_k (r) e^v=0, \hspace{6mm}0<r<e, \hspace{6mm} v\in C^2(0,e)
\end{equation}
satisfying 
\begin{equation}
\label{zimei}
    \liminf_{r\to 0} v(r)>-\infty,
\end{equation}
where $V_k(r)$ is that in \eqref{pote}. Here, we say that $v$ is a radial singular solution of \eqref{apriorieq} if $v$ satisfies \eqref{apriorieq} and $\lim_{r\to 0} v(r)=\infty$.

\subsection{A priori estimates}
We start from introducing the following a 
priori estimates.
\begin{lemma}
\label{apriorileme}
Assume that $v$ is a solution of \eqref{apriorieq} satisfying \eqref{zimei}. Then, there exist $C_1>0$ and $C_2>0$ depending only on $k$ such that 
\begin{equation*}
v(r)\le k\log\left(-\log (r/e)\right) + C_1, \hspace{4mm} 0\le -v'(r) \le -\frac{C_2}{r\log (r/e)} \hspace{4mm} \text{for $0<r<e$},
\end{equation*}
and
\begin{equation}
    \label{aprioriestimate 2-2}
    -rv'(r)= \int_{0}^{r}sV_{k}(s)e^{v(s)}\,ds \hspace{4mm} \text{for $0<r<e$}.
\end{equation}
\end{lemma}
We remark that this lemma can be proved by a similar argument to that in the proof of \cite[Lemma 2.3]{Mi2020}.
For readers convenience, we show the proof.

\begin{proof}
We first prove that $v'\le 0$ in $(0,e)$ by contradiction. Thus, we assume that $v'(t)>0$ for some $t>0$. Since 
\begin{equation}
\label{inte}
    (rv')'=-r V_k(r)e^v\le 0 \hspace{4mm} \text{in $(0,e)$},
\end{equation}
we have $v'(s)\ge \frac{t}{s}v'(t)$ for all $0<s<t$. Integrating this inequality over $(r,t)$, we get $v(r)\le v(t) -tv'(t)\log(t/r)\to-\infty$ as $r\to 0$, which contradicts \eqref{zimei}.

Next, we take $0<r_0\le s\le r<e$. Thanks to the fact that $v'\le 0$  in $(0,e)$,
by integrating \eqref{inte} over $(r_0, s)$, we get 
\begin{align*}
    -sv'(s)&=-r_0v'(r_0)+\int_{r_0}^{s}tV_k(t)e^v\,dt\\
    &\ge \frac{1}{1+k}e^{v(s)}\left((-\log (s/e))^{-(1+k)}-(-\log (r_0/e))^{-(1+k)}\right).
\end{align*}
Letting $r_0\to 0$, we have
\begin{equation*}
-e^{-v(s)}v'(s)\ge \frac{1}{1+k} s^{-1} (-\log (s/e))^{-(1+k)}.
\end{equation*}
Integrating this inequality over $(\rho,r)$ and letting $\rho\to 0$, we get 
\begin{equation*}
   e^{v(r)} \le k(1+k) (-\log (r/e))^{k}. 
\end{equation*}
Hence, we obtain 
\begin{equation}
\label{aiyae}
    v(r)\le k\log(-\log(r/e))+C_1 \hspace{6mm} \text{with $C_1=\log(k(k+1))$}.
\end{equation}
Then, we prove that $rv'(r)\to 0$ as $r\to 0$ by contradiction. Therefore, we assume that $\limsup_{r\to 0} -r v'(r)>c$ for some $c>0$. Since $-rv'$ is non-decreasing, we get $-v'(r)\ge c/r$ for all $r<e$. Therefore, we deduce that $v(r)\ge v(1)-c\log r$, which contradicts \eqref{aiyae}. 

Therefore, by integrating \eqref{inte} over $(r_0,r)$ and then letting $r_0\to 0$, we get \eqref{aprioriestimate 2-2}. Moreover, by using \eqref{aiyae} again, we have 
\begin{equation*}
    -v'(r)\le r^{-1}\int_{0}^{r}sV_k(s)e^v\,ds\le \int_{0}^{r} \frac{e^{C_1}}{rs(-\log(s/e))^{2}}\,ds\le \frac{-e^{C_1}}{r\log(r/e)}.
\end{equation*}
Thus, we get the result.
\end{proof}

As a result of Lemma \ref{apriorileme}, we obtain an existence/uniqueness result for the solution of \eqref{ev}.

\begin{corollary}
\label{formeu}
Let $k>0$ and $\beta\in \mathbb{R}$. Then, the equation \eqref{ev} has a unique solution $v=v(r,\beta)\in C^2(0,e)\cap C^0[0,e)$. Moreover, $v\in H^1(B_{r_0})$ for all $0<r_0<e$. In addition, $v$ satisfies
\begin{equation}
\label{inteqe}
    v(r,\beta)=\beta-\int_{0}^{r} \int_{0}^{s}\frac{t}{s}V_k(t)e^{v(t)}\,dt\,ds.
\end{equation}
\end{corollary}

\begin{proof}
Assume that $v\in C^2(0,e)\cap C^0[0,e)$ is a solution of \eqref{apriorieq}. Let $0<s<r<e$. Then, thanks to Lemma \ref{apriorileme}, we have
\begin{equation*}
    -v'(s)=\int_{0}^{t}\frac{t}{s}V_k(t)e^{v(t)}\,dt.
\end{equation*}
Integrating the above over $(\rho,r)$ and letting $\rho\to 0$, we get \eqref{inteqe}. Moreover, by a standard ODE argument, we deduce that \eqref{inteqe} has a unique solution $v\in C^2(0,e)\cap C^0[0,e)$, which satisfies \eqref{ev}. Finally, we can easily confirm that $v\in H^1(B_{r_0})$ for all $0<r_0<e$.
\end{proof}
As a result of Corollary \ref{formeu}, by using the specific change of variables introduced at the beginning of this section, we show Theorem \ref{diagramth}.
\begin{proof}[Proof of Theorem \ref{diagramth}]
Let $\beta\in \mathbb{R}$. Then, $(\lambda(\beta), u(r,\alpha(\beta))):=(e^{v(1,\beta)}, v(r, \beta)-\log \lambda(\beta))$ is a radial solution of \eqref{gelfand} with $f(u)=e^u$ and $\alpha(\beta):=\beta-\log \lambda(\beta)$. On the other hand, as mentioned at the beginning of this section, we deduce that every radial solution of \eqref{gelfand} is parameterized by $\beta$. In addition, thanks to Corollary \ref{formeu}, we deduce that every solution $u(r,\alpha(\beta))\in H^1_{0}(B_1)$ and $u$ satisfies \eqref{gelfand} in the weak sense. Moreover, by the fact that $\lambda(\beta)=e^{v(1,\beta)}$ and Lemma \ref{apriorileme}, we obtain $\lambda(\beta)\le \lambda^{*}$ with some $\lambda^{*}>0$ depending only on $k$ and thus we verify that the bifurcation curve is emanating from $(0,0)$. Finally, the analyticity of the bifurcation curve follows from the analyticity of $f(u)=e^u$. Thus, we get the result.
\end{proof}

At the end of this subsection, we prove the following estimate for singular solutions.
\begin{lemma}
\label{lowerleme}
Assume that $v$ is a singular 
solution of \eqref{apriorieq}. Then, we have
\begin{equation*}
    \limsup_{r\to 0} \left(v-k\log\left(-\log(r/e)\right)\right)>-\infty.
\end{equation*}
\end{lemma}

\begin{proof}
Assuming $v-k\log\left(-\log(r/e)\right)\to -\infty$ as $r\to 0$, let us derive a contradiction. Let $0<\varepsilon<k/2$. Then, there exists $0<r_0<1$ such that 
\begin{equation*}
   rV_k(r)e^v= \frac{1}{r(\log (r/e))^2}e^{(v-k\log\left(-\log(r/e)\right))}<\frac{\varepsilon}{r(\log (r/e))^2}, \hspace{4mm} 0<r<r_0.
\end{equation*}
Hence, by Lemma \ref{apriorileme}, we have 
\begin{equation*}
 -v'(s)=\frac{1}{s}\int_{0}^{s}tV_k(t)e^{v(t)}\,dt\le \int_{0}^{s}\frac{\varepsilon}{st(\log (t/e))^2}\,dt\le \frac{-\varepsilon}{s\log (s/e)}, \hspace{4mm} 0<s<r_0.
\end{equation*}
Thanks to the above estimate, we get
\begin{equation*}
    \left((-\log (s/e))^{-\varepsilon}e^v\right)'=(-\log (s/e))^{-\varepsilon}e^v(v'-\varepsilon(s\log (s/e))^{-1})\ge 0, \hspace{4mm} 0<s<r_0
\end{equation*}
and thus we deduce that 
\begin{equation*}
e^{v(s)}\le (-\log (r_0 /e))^{-\varepsilon}e^{v(r_0)}(-\log (s/e))^{\varepsilon}\le C (-\log (s/e))^{\varepsilon}, \hspace{4mm} 0<s<r_0,  
\end{equation*}
where $C>0$ is a constant. Therefore, by the fact that $\varepsilon<\frac{k}{2}$ and Lemma \ref{apriorileme}, we have
\begin{equation*}
    -v'(s)=\frac{1}{s} \int_{0}^{s}tV_k(t)e^{v(t)}\,dt\le \int_{0}^{s}\frac{C}{st(-\log (t/e))^{\frac{k}{2}+2}}\,dt\le \frac{C}{s(-\log (t/e))^{\frac{k}{2}+1}}
\end{equation*}
for all $0<s<r_0$. Integrating the above inequality over $(\rho, r_0)$, we get
\begin{equation*}
    v(\rho)-v(r_0)\le \frac{C}{{r_0(-\log (r_0/e))^{\frac{k}{2}}}}, 
\end{equation*}
which contradicts the assumption that $v(r)\to \infty$ as $r\to 0$.
\end{proof}
\subsection{An Emden-Fowler type transformation}
In this subsection, we obtain a new Emden-Fowler type transformation and prove the uniqueness of a singular solution and the oscillation of the bifurcation curve. At first, we mention that 
\begin{equation*}
    W(r)=k \log\left(-\log (r/e)\right)+\log k
\end{equation*}
is a singular solution of \eqref{apriorieq}.

Let $v$ be a solution of \eqref{apriorieq} 
satisfying \eqref{zimei}.
For any $r\in (0,e)$, we apply the following Emden-Fowler type transformation  
 \begin{equation*}
 \label{transe}
 w(t)=v(r)-W(r)\hspace{4mm}\text{with} \hspace{2mm}  t=\log(-\log (r/e)).  
 \end{equation*}
Then, $w$ satisfies 
\begin{equation}
\label{ew}
    \frac{d^2}{dt^2}w-\frac{d}{dt}w+k(e^w-1)=0, \hspace{4mm} t\in \mathbb{R}.
\end{equation}
Then, we observe the linearized equation
\begin{equation}
\label{lineae}
    \frac{d^2}{dt^2}w-\frac{d}{dt}w+kw=0.
\end{equation}
The associated eigenvalues are given by
\begin{equation*}
    \lambda_{\pm}=\frac{1}{2}\left(1\pm \sqrt{1-4k}\right).
\end{equation*}
Hence, all nontrivial solutions of \eqref{lineae} change sign infinitely many times provided $k>1/4$. Applying the Sturm's comparison theorem (see Lemma \ref{strumlem}), we have the following

\begin{proposition}
\label{osceu}
All nontrivial solutions of \eqref{ew} satisfying $w\to 0$ as $t\to -\infty$
change sign infinitely many times when $k>1/4$.
\end{proposition}

\begin{proof}[Proof of Theorem \ref{singularth} \rm{(i)} \textit{and Theorem \ref{mainthe}} \rm{(i)}]
We first prove the uniqueness of a singular solution.
Let $(\lambda_{*}, U_{*})$ be a radial singular solution of \eqref{gelfand} for $f(u)=e^u$. 
Then, $v:=U_{*}+\log \lambda_{*}$ is a singular solution of \eqref{apriorieq} satisfying $v(1)=\log \lambda_{*}$. Moreover, we define $w(t):=v(r)-W(r)$ with $t=\log(-\log (r/e))$. Thanks to Lemma \ref{apriorileme} and Lemma \ref{lowerleme}, we have 
\begin{equation*}
-\infty<\limsup_{t\to\infty} w(t)\le C,
\end{equation*}
where $C>0$ is depending only on $k$. Therefore, by \cite[Lemma 3.2]{Mi2020}, we get $w(t)\to 0$ as $t\to\infty$. Moreover, since the real parts of the associated eigenvalues of \eqref{lineae} are positive, we get $w(t)=0$. In particular, since $v(1)=W(1)$, it follows that $(\lambda_{*},U_{*})=(\log k, W-\log k)$. Moreover, we can confirm that $U_{*}\in H^1_{0}(B_1)$.

Next, we show that the bifurcation curve converges to the singular solution. Let $(\lambda,u)=(\lambda(\beta),u(r,\alpha(\beta)))$ be a radial solution of \eqref{gelfand} for $f(u)=e^u$. We define $v=u+\log\lambda$ and $w(t)=v(r)-W(r)$. In addition, we define $\hat{w}(s):=w(t)=w(t,\beta)$ with $s=t-\frac{\beta}{k}+\frac{\log k}{k}$. Then, $\hat{w}$ satisfies \eqref{ew} and the initial condition $\lim_{s\to\infty}(\hat{w}+ks)=0$. Moreover, by Corollary \ref{formeu}, we deduce that $\hat{w}(s)$ is independent of $\beta$. Here, we observe the following Lyapunov function
\begin{equation*}
    \mathcal{L}(\hat{w}(t)):=\frac{1}{2}(\hat{w}')^2+k(e^{\hat{w}}-\hat{w}).
\end{equation*}
Then, we verify that this function is non-decreasing and thus $\hat{w}$ and  $\hat{w}'$ remain bounded as $t\to-\infty$. 
Since $v(r,\beta)=\hat{w}(t-\frac{\beta}{k}+\frac{\log k}{k})+W(r)$, for each $\varepsilon>0$, we obtain
$|v(r,\beta)|, |v'(r,\beta)|<C(\varepsilon)$
for all $r\in[\varepsilon, 2]$ and $\beta>1$, where $C(\varepsilon)>0$ is depending only on $\varepsilon$ and $k$. In addition, for each $\varepsilon>0$, there exists some $\beta(\varepsilon)\in \mathbb{R}$ depending only on $\varepsilon$ and $k$ such that $|v(r,\beta)-W(r)|<C$
for all $r\in[\varepsilon, 2]$ and $\beta>\beta(\varepsilon)$, where $C>0$ is depending only on $k$. Therefore, by the elliptic regularity theory (see \cite{gil}), Arzel\`a-Ascoli theorem, and a diagonal argument, there exist a sequence $\{\beta_n\}_{n\in \mathbb{N}}$ and a singular solution $V\in C^2(0,1]$ of \eqref{ev} such that $\beta_n\to\infty$ and $v(r,\beta_n)\to V$ in $C^2_{\mathrm{loc}}(0,1]$ as $n\to\infty$. By the uniqueness of a singular solution, it follows that $V=W$ and thus we get the result. 

Finally, we prove the oscillation of the bifurcation curve. Thanks to the above argument, we have $\hat{w}(s)\to 0$ as $s\to-\infty$. Since
\begin{equation*}
\lambda(\beta)=e^{W(1)+w(0,\beta)}=ke^{\hat{w}(-\frac{\beta}{k}+\frac{\log k}{k})},
\end{equation*}
by Proposition \ref{osceu}, we deduce that $\lambda(\beta)$ turns around $k$ infinitely many times when $k>1/4$. 
\end{proof}

\subsection{Stability of singular solutions}
In this subsection, we study the stability of the singular solution $W$. From the stability of $W$, we obtain the separation property. As a result, we prove Theorem \ref{mainthe} (ii) and Theorem \ref{sepathm} (i). In order to study the stability of $W$, we introduce the following 
\begin{proposition}
\label{critical prop}
Let $R\ge 1$. Then, for any $\varphi\in C^{0,1}_{0}(B_1)$, we have
\begin{equation}
\label{criticalhardy}
    \frac{1}{4}\int_{B_1}\frac{\varphi^2}{|x|^2(\log(R/|x|))^2}\,dx \le \int_{B_1}|\nabla \varphi|^2\,dx.
\end{equation}
Moreover, for any $\varepsilon>0$, there exists a sequence $\{\varphi_n\}_{n\in \mathbb{N}}\subset C^{0,1}_{0}(B_1)$ such that 
$\mathrm{supp}(\varphi_i)\cap \mathrm{supp}(\varphi_j)=\emptyset$ for $i\neq j$ and
\begin{equation*}
    \frac{1+\varepsilon}{4}\int_{B_1}\frac{\varphi^2_{i}}{|x|^2(\log(R/|x|))^2}\,dx > \int_{B_1}|\nabla \varphi_{i}|^2\,dx \hspace{4mm} \text{for any $i\in \mathbb{N}$}.
\end{equation*}
\end{proposition}
We mention that \eqref{criticalhardy} is well-known and called the critical Hardy inequality (see \cite{II, ST} and the references therein).

\begin{proof}
For $\varphi\in C^{0,1}_{0}(B_1)$, we define $\xi=(\log(R/|x|))^{-\frac{1}{2}}\varphi$. Since $\xi(0)=0$ and 
\begin{align*}
    |\nabla \varphi|^2&=|(\log(R/|x|))^{\frac{1}{2}}\nabla \xi-\frac{1}{2}(\log(R/|x|))^{-\frac{1}{2}}\frac{x}{|x|^2}\xi|^2\\
    &=\frac{\xi^2}{4|x|^2 (\log(R/|x|))}+(\log(R/|x|))|\nabla\xi|^2- \frac{x}{2|x|^2}\cdot\nabla (\xi^2)\\
    &=\frac{\varphi^2}{4|x|^2 (\log(R/|x|))^2}+(\log(R/|x|))|\nabla\xi|^2- \frac{x}{2|x|^2}\cdot\nabla (\xi^2),
\end{align*}
we get
\begin{align*}
&\int_{B_1}|\nabla \varphi|^2\,dx-\frac{1+\varepsilon}{4}\int_{B_1}\frac{\varphi^2}{|x|^2(\log(R/|x|))^2}\,dx\\
&=\int_{B_1}(\log(R/|x|))|\nabla\xi|^2\,dx
-\int_{B_1}\frac{\varepsilon\xi^2}{4|x|^2 (\log(R/|x|))}\,dx-\int_{B_1}\frac{x}{2|x|^2}\cdot\nabla (\xi^2)\,dx\\
&=\int_{B_1}(\log(R/|x|))|\nabla\xi|^2\,dx
-\int_{B_1}\frac{\varepsilon\xi^2}{4|x|^2 (\log(R/|x|))}\,dx.
\end{align*}
Thus, we get \eqref{criticalhardy} by setting $\varepsilon=0$. In addition, for any $0<\varepsilon<1$, we take $\xi_{n}(r)=\chi_{[r_{n+1},r_{n}]}(r)\sin t$
with $r=|x|$, $t=\frac{\varepsilon}{2}\log(\log(R/r))$, and 
$r_{n}=Re^{-e^{2\pi n/\varepsilon}}$. Then,
it follows that $\varphi_n:=(\log(R/|x|))^{\frac{1}{2}}\xi_n \in C^{0,1}(B_1)$ and 
\begin{align*}
\int_{B_1}|\nabla \varphi_n|^2\,dx&-\frac{1+\varepsilon}{4}\int_{B_1}\frac{\varphi_{n}^2}{|x|^2(\log(R/|x|))^2}\,dx\\
   &=\int_{B_1}(\log(R/|x|))|\nabla\xi_{n}|^2\,dx
-\int_{B_1}\frac{\varepsilon\xi_{n}^2}{4|x|^2 (\log(R/|x|))}\,dx\\
&=\varepsilon\pi\int_{n\pi}^{(n+1)\pi}\cos^{2} t\,dt - \pi\int_{n\pi}^{(n+1)\pi}\sin^{2} t\,dt<0.
\end{align*}
Thus, we get the result.
\end{proof}

\begin{proof}[Proof of Theorem \ref{morseth} \rm{(i)}]
Take $\varphi\in C^{0,1}_{0}(B_1)$. Then, we deduce that 
\begin{align*}
\mathcal{Q}_{U_{*}}(\varphi)=\int_{B_1}|\nabla \varphi|^2\,dx-\int_{B_1} \lambda_{*}&V_ke^{U_{*}}\varphi^2\,dx= \int_{B_1}|\nabla \varphi|^2\,dx-\int_{B_1}V_k e^{W}\varphi^2\,dx\\
   &=\int_{B_1}|\nabla \varphi|^2\,dx-\int_{B_1}\frac{k}{|x|^2(\log(e/|x|))^2}\varphi^2\,dx.
\end{align*}
Therefore, thanks to Proposition \ref{critical prop}, we get the result.
\end{proof}
\begin{remark}
\label{stre}
\rm{By a similar method to that in the proof of Theorem \ref{morseth} (i), we deduce that the singular solution $W$ is stable in $B_e$ if and only if $k\le 1/4$.}
\end{remark}

Next, we prove the following separation result, which plays a key role in studying the bifurcation structure. 
\begin{proposition}
\label{sepae}
Assume that $k\le\frac{1}{4}$. Let $\beta<\gamma$ and $v(r,\beta)$ be the solution of \eqref{ev}. Then, we have
\begin{itemize}
    \item[{\rm{(i)}}] $v(r,\beta)<W(r)$ in $(0,e)$.
    \item [{\rm{(ii)}}] $v(r,\beta)<v(r,\gamma)$ in $(0,e)$. 
\end{itemize}
\end{proposition}

\begin{proof}
We first prove (i) by contradiction. Thus, we assume that there exists $r_0\in (0,e)$ such that $W(r)-v(r,\beta)>0$ in $(0,r_0)$ and $W(r_0)=v(r_0,\beta)$. We define 
\begin{equation*}
\hat{v}(r)=
\begin{cases}
    W(r)-v(r,\beta) &\text{if $0\le r<r_0$,}\\
    0               &\text{otherwise.}
\end{cases}
\end{equation*}
Since $W$ and $v$ satisfy \eqref{ev} and $v, W \in H^1(B_{r_0})$, we have
\begin{equation*}
    \int_{B_{r_0}}|\nabla \hat{v}|^2\,dx= \int_{B_{r_0}}V_k(|x|)(e^{W}-e^{v})\hat{v}\,dx\le \int_{B_{r_0}}V_k(|x|)e^{W}\hat{v}^2\,dx.
\end{equation*}
On the other hand, by Remark \ref{stre}, we get
\begin{equation*}
\int_{B_{r_0}}V_k(|x|)e^{W}\hat{v}^2\,dx\le \int_{B_{r_0}}|\nabla \hat{v}|^2\,dx.
\end{equation*}
Therefore, we get $\hat{v}=0$, which is a contradiction.

Moreover, thanks to the above assertion, we deduce that $v(r,\gamma)$ is stable for all $\gamma\in \mathbb{R}$. Hence, by using a method similar to the proof of (i), we get (ii). 
\end{proof}

\begin{proof}[Proof of Theorem \ref{mainthe} \rm{(ii)}]
Let us denote by $\cdot$ the differentiation with respect to $\beta$. By Proposition \ref{sepae}, we have $\dot{\lambda}\ge 0$ for all $\beta\in \mathbb{R}$. Thus, it suffices to prove $\dot{\alpha}>0$ for all $\beta\in \mathbb{R}$. 
We prove the assertion by contradiction. Thus, we assume that $\dot{\alpha}(\beta_0)\le 0$ for some $\beta_0 \in \mathbb{R}$.
Since $\alpha=\beta-\log \lambda$, we get $\dot{\lambda}(\beta_0)>0$. By differentiating \eqref{gelfand} with respect to $\beta$, we have
\begin{equation*}
\left\{
\begin{alignedat}{4}
 -\Delta \dot{u}&=\lambda V_k(|x|)e^u\dot{u}+ \dot{\lambda}V_k(|x|)e^u\hspace{4mm} \text{in } B_1,\\
\dot{u}(0)&=\dot{\alpha}, \hspace{4mm} \dot{u}(1)=0,\hspace{4mm}\dot{u}\in C^{0}[0,1]\cap C^{2}_{\mathrm{loc}}(0,1]\cap H^1(B_1).
\end{alignedat}
\right.
\end{equation*}
Hence, by using a similar argument to that in the proof of Lemma \ref{apriorileme} if $\dot{\alpha}(\beta_0)=0$, we deduce that there exists $0<r_0\le 1$ such that $\dot{u}(r,\beta_0)<0$ in $(0,r_0)$ and $\dot{u}(r_0,\beta_0)=0$. We define 
\begin{equation*}
    \hat{u}(r)=
\begin{cases}
    \dot{u}(r,\beta_0) &\text{if $0\le r<r_0$,}\\
    0              &\text{otherwise.}
\end{cases}
\end{equation*}
Since $\dot{\lambda}(\beta_0)>0$, we have 
\begin{align*}
    \int_{B_{1}}|\nabla \hat{u}|^2\,dx&= \int_{B_{1}}\lambda(\beta_0) V_k(|x|)e^{u}\hat{u}^2\,dx+ \dot{\lambda}(\beta_0)\int_{B_{1}}V_k(|x|)e^{u}\hat{u}\,dx\\
    &<\int_{B_{1}}\lambda(\beta_0) V_k(|x|)e^{u}\hat{u}^2\,dx.
\end{align*}
On the other hand, by Proposition \ref{sepae}, we get
\begin{align*}
\int_{B_{1}}\lambda(\beta_0) V_k(|x|)e^{u}\hat{u}^2\,dx= \int_{B_{1}} V_k(|x|)e^{v}\hat{u}^2\,dx&\le \int_{B_{1}} V_k(|x|)e^{W}\hat{u}^2\,dx\\
&\le \int_{B_{1}}|\nabla \hat{u}|^2\,dx, 
\end{align*}
which is a contradiction.
\end{proof}

Finally, we prove Theorem \ref{sepathm} (i).
\begin{proof}[Proof of Theorem \ref{sepathm} \rm{(i)}]
We remark that the result follows from Proposition \ref{sepae} (ii) in the case $k\le \frac{1}{4}$. Thus, it remains the case $k>\frac{1}{4}$. For $0<\beta<\gamma$, we define 
$w(t,\beta)=v(r,\beta)-W(r)$ with $t=\log (-\log(r/e))$. Here, we remark that $w(t,\beta)\to 0$ as $t\to -\infty$ for all $\beta>0$ by the proof of Theorem \ref{singularth} (i). In addition, we define $w_0(t)=w(t,\gamma)-w(t,\beta)$. Then
$w_0$ satisfies
\begin{equation*}
    \frac{d^2}{dt^2}w_0 -\frac{d}{dt}w_0+ k g(t) w_0, \hspace{4mm} t\in \mathbb{R}
\end{equation*}
with
\begin{equation*}
    g(t)=\frac{e^{w(t,\gamma)}-e^{w(t,\beta)}}{w(t,\gamma)-w(t,\beta)}\in C^0(\mathbb{R}).
\end{equation*}
Then, by the fact that $w(t,\beta)\to 0$ as $t\to -\infty$ for all $\beta>0$, we get $kg(t)\to k$ as $t\to -\infty$. Moreover,
we remind that all nontrivial solutions of \eqref{lineae} change sign infinitely many times provided $k>1/4$. Applying the Sturm's comparison theorem (see Lemma \ref{strumlem}), we get the result.
\end{proof}

\section{Power case}
\label{powersec}
In this section, we deal with the case $f(u)=(1+u)^p$. 
We begin by introducing a specific change of variables which is used in \cite{korman2018}. Let $(\lambda,u)$ be a radial solution of \eqref{gelfand} with $\lVert u \rVert_{L^\infty(B_1)}=\alpha$. We define 
$v:=\lambda^{\frac{1}{p-1}}(u+1)$ and $\beta:=\lambda^{\frac{1}{p-1}}(\alpha+1)$. Then, $v$ can be extended on $(0,e)$ such that $v$ satisfies 
\begin{equation}
\label{vp}
\left\{
\begin{alignedat}{4}
&v''+  \frac{1}{r}v' + V_k (r) |v|^p=0, \hspace{4mm}0<r<e,\\
&v(0)=\beta, \hspace{6mm}v\in C^2(0,e)\cap C^0[0,e).
\end{alignedat}
\right.
\end{equation}

In the following, we consider a solution of the equation
\begin{equation}
\label{apriorieqp}
v''+  \frac{1}{r}v' + V_k (r) |v|^p =0, \hspace{4mm}0<r<e, \hspace{4mm} v\in C^2(0,e)
\end{equation}
satisfying
\begin{equation}
\label{zimeip}
    \liminf_{r\to 0} v(r)>0,
\end{equation}
where $V_k(r)$ is that in \eqref{pote}. Here, we say that $v$ is a singular solution of \eqref{apriorieqp} if $v$ satisfies \eqref{apriorieqp} and $\lim_{r\to 0} v(r)=\infty$.

\subsection{A priori estimates}
We begin by introducing the following a priori estimates.
\begin{lemma}
\label{apriorilemmap}
Assume that $v$ is a solution of \eqref{apriorieqp} satisfying \eqref{zimeip}. Then, there exist $C_1>0$ and $C_2>0$ depending only on $k$ and $p$ such that 
\begin{equation*}
\label{aprioriestimate 3-1}
    v(r)\le C_1\left(-\log (r/e)\right)^{\frac{k}{p-1}} \hspace{2mm}\text{and} \hspace{2mm} 0\le -v'(r) \le  \frac{C_2}{r} \left(-\log (r/e)\right)^{\frac{k-p+1}{p-1}}
\end{equation*}
for all $r\in (0,e)$ satisfying $v(r)>0$. Moreover, it follows that
\begin{equation}
    \label{aprioriestimate 3-2}
    -rv'(r)= \int_{0}^{r}sV_{k}(s)|v(s)|^{p}\,ds \hspace{4mm} \text{for all $0<r<e$}.
\end{equation}
\end{lemma}
We remark that this lemma can be proved by a similar argument to that in the proof of \cite[Lemma 2.1]{Mi2020}. For readers convenience, we show the proof.
\begin{proof}
Since 
\begin{equation}
\label{intep}
    (rv')'=-r V_k(r)|v|^p\le 0 \hspace{4mm}\text{in $(0,e)$,}
\end{equation}
we obtain that $v'\le 0$ in $(0,e)$ by the same argument as in the proof of Lemma \ref{apriorileme}. Next, we take $0<r_0\le s\le r<e$ such that $v(r)>0$. Thanks to the fact that $v'\le 0$ in $(0,e)$, by integrating \eqref{intep} over $(r_0, r)$, we get 
\begin{align*}
    -sv'(s)&=-r_0v'(r_0)+\int_{r_0}^{s}tV_k(t)|v|^p\,dt\\
    &\ge \frac{1}{1+k}|v(s)|^p\left((-\log (s/e))^{-(1+k)}-(-\log (r_0/e))^{-(1+k)}\right).
\end{align*}
Letting $r_0\to 0$, we have
\begin{equation*}
-v^{-p}(s) v'(s)\ge \frac{1}{1+k} s^{-1} (-\log (s/e))^{-(1+k)}.
\end{equation*}
Integrating the above inequality over $(\rho,r)$ and then letting $\rho\to 0$, we get
\begin{equation*}
   v^{p-1}(r) \le \frac{k(1+k)}{p-1} (-\log (r/e))^{k}. 
\end{equation*}
Hence, we obtain 
\begin{equation}
\label{aiyap}
    v(r)\le C_1\left(-\log (r/e)\right)^{\frac{k}{p-1}}
\end{equation}
with some $C_1>0$ depending only on $k$ and $p$. Thus, we obtain that $rv'(r)\to 0$ as $r\to 0$ by the 
same method as in the proof of Lemma \ref{apriorileme}. Therefore, by integrating \eqref{intep} over $(r_0,r)$ and then letting $r_0\to 0$, we get \eqref{aprioriestimate 3-2}. Moreover, by using \eqref{aiyap} again, we have 
\begin{align*}
    -v'(r)\le r^{-1}\int_{0}^{r}sV_k(s)|v|^p\,ds&\le \int_{0}^{r} \frac{C_{1}^p}{rs(-\log(s/e))^{-\frac{k-p+1}{p-1}+1}}\,ds\\
    &\le \frac{C_2}{r}(-\log(r/e))^{\frac{k-p+1}{p-1}}
\end{align*}
for all $0<r<e$ satisfying $v(r)>0$, where $C_2>0$ depends only on $k$ and $p$. 
Thus, we get the result.
\end{proof}
As a result of Lemma \ref{apriorilemmap}, we obtain an existence/uniqueness result for the solution of \eqref{vp}.

\begin{corollary}
\label{formup}
Let $k>0$ and $\beta>0$. Then, the equation \eqref{vp} has a unique solution $v=v(r,\beta)\in C^2(0,e)\cap C^0[0,e)$. Moreover, $v\in H^1(B_{r_0})$ for all $0<r_0<e$. In addition, $v$ satisfies
\begin{equation*}
\label{inteqp}
    v(r,\beta)=\beta-\int_{0}^{r} \int_{0}^{s}\frac{t}{s}V_k(t)|v(t)|^{p}\,dt\,ds.
\end{equation*}
\end{corollary}
\begin{proof}
Thanks to Lemma \ref{apriorilemmap}, we can get the result by using a similar argument to that in the proof of Corollary \ref{formeu}.
\end{proof}

\subsection{An Emden-Fowler type transformation}
In this section, we obtain a new Emden-Fowler type transformation and prove the uniqueness of a singular solution and the oscillation of the bifurcation curve. At first, we mention that 
\begin{equation*}
    W(r)= \left(\theta(1-\theta)\right)^{\frac{\theta}{k}}(-\log (r/e))^\theta \hspace{4mm}\text{with} \hspace{4mm} \theta=\frac{k}{p-1} 
\end{equation*}
is a singular solution of \eqref{apriorieqp} provided $p>p_s:=k+1$.

Let $v$ be a solution of \eqref{apriorieqp} 
satisfying \eqref{zimeip}.
For any $r\in (0,e)$, we apply the Emden-Fowler type transformation  
 \begin{equation}
 \label{transep}
 w(t)=W^{-1}(r)v(r)-1\hspace{4mm}\text{with} \hspace{2mm}  t=\log(-\log (r/e)). 
 \end{equation}
Then, $w$ satisfies 
\begin{equation}
\label{pw}
    \frac{d^2}{dt^2}w+(2\theta-1)\frac{d}{dt}w+\theta(1-\theta)(|w+1|^p-(w+1))=0, \hspace{4mm} t\in \mathbb{R}.
\end{equation}
By using the transformation, we prove the following estimate for singular solutions.
\begin{lemma}
\label{lowerlemp}
Assume that $p>p_c:=2k+1$. Let $v$ be a singular solution of \eqref{apriorieqp}. Then, we have 
\begin{equation*}
    \limsup_{r\to 0} (-\log (r/e))^{-{\frac{k}{p-1}}}v>0.
\end{equation*}
\end{lemma}

\begin{proof}
We prove the assertion by contradiction. Therefore, we assume that $(-\log (r/e))^{-\theta}v\to 0$ as $r\to 0$ with $\theta=\frac{k}{p-1}$.
Then, we claim that there exists $r_1>0$ such that $v(r)>0$ for $0<r<r_1$ and 
\begin{equation}
\label{asyeq}
    \left((-\log (r/e))^{-\theta}v\right)' \ge 0 \hspace{4mm} \text{for all $0<r<r_1$.}
\end{equation}
Indeed, we define $w$ as that in \eqref{transep}. Then, it follows from the above assumption that $w(t)\to -1$ as $t\to\infty$. Moreover,
since $w$ satisfies \eqref{pw}, we have $w(t)>-1$ for $t_1<t$ and 
\begin{equation*}
    (e^{(2\theta-1)t}w')'=-e^{(2\theta-1)t}\theta(1-\theta)(|w+1|^p-(w+1))\ge0 \hspace{4mm}\text{for $t_1<t$}
\end{equation*}
with some large $t_1$. Therefore, by using a similar argument to that in the proof of Lemma \ref{apriorileme}, we get
\eqref{asyeq}.

Let $0<\varepsilon<\frac{k(1-\theta)}{2p}$. Then, there exists $0<r_0<r_1$ such that 
\begin{equation*}
   rV_k(r)|v|^{p-1}= \frac{1}{r(\log (r/e))^2} 
   |(-\log (r/e))^{-\theta}v|^{p-1}
   <\frac{\varepsilon}{r(\log (r/e))^2}, \hspace{1mm} 0<r<r_0.
\end{equation*}
Hence, by Lemma \ref{apriorilemmap} and \eqref{asyeq}, we have 
\begin{align*}
 -v'&(s)=\frac{1}{s} \int_{0}^{s}tV_k(t)|v(t)|^p\,dt\le \frac{\varepsilon}{s}\int_{0}^{s}\frac{v}{t(\log (t/e))^2}\,dt\\
 &\le \frac{\varepsilon v(s)}{s(-\log (s/e))^{\theta}}
\int_{0}^{s}\frac{1}{t(-\log (t/e))^{2-\theta}}\,dt\le 
 \frac{-\varepsilon v(s)}{(1-\theta)s\log (s/e)}, \hspace{1mm} 0<s<r_0.
\end{align*}
Thanks to the above estimate, we get
\begin{equation*}
    \left((-\log (s/e))^{-\delta}v\right)'=(-\log (s/e))^{-\delta}(v'-\delta v(s\log (s/e))^{-1})\ge 0, \hspace{4mm} 0<s<r_0,
\end{equation*}
where $\delta=\frac{\varepsilon}{1-\theta}$. Therefore, we deduce that 
\begin{equation*}
|v(s)|^p\le (-\log (r_0 /e))^{-p\delta}v(r_0)^p(-\log (s/e))^{p\delta}\le C (-\log (s/e))^{p\delta}, \hspace{2mm} 0<s<r_0,  
\end{equation*}
where $C>0$ is a constant. Accordingly, by using Lemma \ref{apriorilemmap} again, we have
\begin{align*}
    -v'(s)=\int_{0}^{s}\frac{t}{s}V_k(t)|v(t)|^p\,dt&\le \int_{0}^{s}\frac{C}{st(-\log (t/e))^{\frac{k}{2}+2}}\,dt
    \le \frac{C}{s(-\log (t/e))^{\frac{k}{2}+1}}.
\end{align*}
Integrating the above inequality over $(\rho, r_0)$, we get
\begin{equation*}
    v(\rho)-v(r_0)\le \frac{C}{{r_0(-\log (r_0/e))^{\frac{k}{2}}}}, 
\end{equation*}
which contradicts the assumption that $v(r)\to \infty$ as $r\to 0$.
\end{proof}

Next, we observe the linearized equation
\begin{equation}
\label{lineap}
    \frac{d^2}{dt^2}w+(2\theta-1)\frac{d}{dt}w+k(1-\theta)w=0.
\end{equation}
The associated eigenvalues are given by
\begin{equation*}
    \lambda_{\pm}=\frac{1}{2}\left((1-2\theta)\pm \sqrt{4\theta^2+4(k-1)\theta+1-4k}\right).
\end{equation*}
Hence, all nontrivial solutions of \eqref{lineap} change sign infinitely many times provided 
\begin{equation*}
    4\theta^2+4(k-1)\theta+1-4k<0 \hspace{6mm}\iff \hspace{6mm} p^{-}_{\mathrm{JL}}<p<p^{+}_{\mathrm{JL}}.
\end{equation*}
Applying the Sturm's comparison theorem (see Lemma \ref{strumlem}), we have the following 
\begin{proposition}
\label{oscup}
All nontrivial solutions of \eqref{pw} satisfying $w\to 0$ as $t\to -\infty$ change sign infinitely many times when $p^{-}_{\mathrm{JL}}<p<p^{+}_{\mathrm{JL}}$.
\end{proposition}
Let $v=v(r,\beta)$ be a solution of \eqref{vp} with $\beta>0$. We define $\hat{w}(s):=w(t)=w(t,\beta)$ with $s=t-\frac{\log \beta}{\theta}$, where $w$ is that in \eqref{transep}.
Then, $\hat{w}$ is a solution of \eqref{pw} satisfying the initial condition 
\begin{equation}
\label{initialp}
\lim_{s\to\infty}e^{s\theta}(\hat{w}+1)=\left(\theta(1-\theta)\right)^{-\theta/k}.    
\end{equation}
By Corollary \ref{formup}, we verify that $\hat{w}$ is independent of $\beta$.
Then, we have the following
\begin{proposition}
\label{importantprop}
Let $\hat{w}$ be a solution of \eqref{pw}
satisfying \eqref{initialp}. Then,
\begin{itemize}
    \item[{\rm{(i)}}] when $p\ge p_c$, we have $\hat{w}>-1$ in $\mathbb{R}$. Moreover, $\hat{w}$ and $\hat{w}'$ remain bounded as $t\to -\infty$. In addition,  $\liminf_{t\to-\infty}\hat{w}(t)=c-1$ for some $c>0$ if $p>p_c$.
    
    \item[{\rm{(ii)}}] when $p_s<p<p_c$, there exists $-\infty<t_0<t_1<\infty$ such that $\hat{w}'(t)<0$ for $t_1<t$, $\hat{w}'(t_1)=0$, $\hat{w}'(t)>0$ for $t_0\le t<t_1$, $\hat{w}(t)>-1$ for $t_0<t$, and $\hat{w}'(t_0)=-1$.
    \item [{\rm{(iii)}}] when $p=p_c$, there exists $-\infty<t_1<\infty$ such that $\hat{w}'(t)<0$ for $t_1<t$, $\hat{w}'(t_1)=0$, $\hat{w}'(t)>0$ for $t<t_1$, and $\lim_{t\to-\infty}\hat{w}(t)=-1$.
    \end{itemize}
\end{proposition}

\begin{proof}
We define the Lyapunov function
\begin{equation*}
\mathcal{L}(\hat{w}(t)):= \frac{1}{2}(\hat{w}')^2 + \theta(1-\theta)\left(\frac{1}{p+1}|\hat{w}+1|^p(\hat{w}+1)-\frac{1}{2}(\hat{w}+1)^2\right).
\end{equation*}
We remark that $2\theta-1\le 0$ if and only if $p>p_c:=2k+1$. Thus, we deduce that $L$ is increasing if $p>p_c$, constant if $p=p_c$, and decreasing if $p<p_c$. Moreover, we have $L(t)\to 0$ as $t\to\infty$ by the initial condition. Therefore, $\hat{w}(t)>-1$ for all $t\in \mathbb{R}$, and $\hat{w}$ and $\hat{w}'$ remain bounded as $t\to -\infty$ when $p\ge p_c$. In particular, when $p>p_c$
since $L$ is increasing, we have $\liminf_{t\to-\infty}\hat{w}>-1$. Thus, we get the assertion (i).

Next, we consider the case $p\le p_c$. In this case, since $L$ is non-increasing, we deduce that $w'(t)\neq 0$ if $-1<w(t)\le 0$.
We define 
\begin{equation*}
t_1=\inf\{t; \hat{w}'(s)<0 \hspace{2mm}\text{for $s>t$}\}\in [-\infty,\infty).     
\end{equation*}
Here, we prove that $t_1>-\infty$ by contradiction. Thus, we assume that $t_1=-\infty$. Then, we verify that $v(1,\beta)>0$ for all $\beta$.
Since $\hat{w}(-\frac{\log \beta}{\theta})=W^{-1}(1)v(1,\beta)-1$, 
by Lemma \ref{apriorilemmap}, we get $\lim_{t\to-\infty}\hat{w}=C<\infty$. In particular, it follows that $\liminf_{t\to-\infty}|\hat{w}'|=0$. As a result, by the fact that $L$ is non-increasing, we have $C>0$. Moreover, we claim that $\limsup_{t\to-\infty}|\hat{w}'|=0$. Indeed,
if $\limsup_{t\to-\infty}|\hat{w}'|>0$, 
there exists a sequence $\{t_n\}_{n\in \mathbb{N}}$ satisfying $t_n\to -\infty$ as $n\to \infty$ such that $\hat{w}''(t_n)=0$ and $\hat{w}(t_n)\to C$, $\hat{w}'(t_n)\to 0$ as $n\to \infty$, which contradicts the fact that $\hat{w}$ satisfies $\eqref{pw}$ and $C>0$. Therefore, by using $\eqref{pw}$ again, we have 
\begin{equation*}
   \lim_{t\to -\infty} \hat{w}''(t)= -\theta(1-\theta)((C+1)^p-(C+1))<0,
\end{equation*}
which contradicts that $\hat{w}'(t)\to 0$ as $t\to -\infty$. Thus, we get $t_1>-\infty$.

Then, since $L$ is non-increasing, it follows that $0\le L(t_1)\le L(t)$ for $t<t_1$ and $w(t_1)>0$. Moreover, since $\hat{w}$ satisfies \eqref{pw}, we have $\hat{w}''(t_1)<0$. Therefore,
we verify that if $t<t_1$ and $w(s)>-1$ for all $s>t$, then we have
$\hat{w}'(t)>0$ and $\hat{w}(t)<\hat{w}(t_1)$. In addition, when $p=p_c$, since $\hat{w}(t)>-1$ and $L(t)\ge 0$
for all $t\in \mathbb{R}$, we get $\lim_{t\to-\infty} \hat{w}(t)=-1$. On the other hand, when $p<p_c$, since $L$ is decreasing, we have $|\hat{w}'|(t)>c>0$
for some $c>0$ for all $t<t_1$ satisfying $-1\le w(t)\le 0$. Therefore, we get the result.
\end{proof}

\begin{proof}[Proof of Theorem \ref{diagramthp} and Theorem \ref{mainthp} \rm{(i)}]
We first show Theorem \ref{diagramthp}. Assume that $v(r,\beta)$ is a solution of \eqref{vp} with $\beta>0$. Then, it follows that
$\hat{w}(t-\frac{\log \beta}{\theta})=v(r,\beta)W^{-1}(r)-1$ with $t=\log(-\log (r/e))$. Thus, by Proposition \ref{importantprop}, there exists $\beta^{*}\in (0,\infty]$ such that $v(r,\beta)>0$ for $r\in (0,1]$ if and only if $0<\beta<\beta^{*}$. Proposition \ref{importantprop} also tells us that $\beta^{*}=\infty$ if and only if $p_c\le p$. Thus, for $0<\beta<\beta^{*}$, we deduce that $(\lambda(\beta), u(r,\alpha(\beta))):=(v^{p-1}(1,\beta), v^{-1}(1,\beta)v(r, \beta)-1)$ is a radial solution of \eqref{gelfand} with $f(u)=(1+u)^p$ and $\alpha(\beta)=\lambda(\beta)^{-\frac{1}{p-1}}\beta-1$. On the other hand, as mentioned at the beginning of this section,
every radial solution of \eqref{gelfand} with $f(u)=(1+u)^p$ is parameterized by $\beta$ for some $\beta\in (0,\beta^{*})$. 
Moreover, thanks to Corollary \ref{formup}, we deduce that every solution $u(r,\alpha(\beta))\in H^1_{0}(B_1)$ and $u$ satisfies \eqref{gelfand} in the weak sense. In addition, thanks to Lemma \ref{apriorilemmap} and the fact that $\lambda(\beta)=v^{p-1}(1,\beta)$, we obtain $\lambda(\beta)\le \lambda^{*}$ with some $\lambda^{*}>0$ depending only on $k$ and $p$. As a result, we verify that the bifurcation curve is emanating from $(0,0)$. Finally, the analyticity of the bifurcation curve follows from the analyticity of $f(u)=(1+u)^p$. Therefore, we get the result.

We now prove Theorem \ref{mainthp} (i). Let us denote by $\cdot$ the differentiation with respect to $\beta$. Then, for $0<\beta<\beta^{*}$, we have
\begin{equation*}
\dot{\lambda}=(p-1)v^{p-2}(1,\beta)\dot{v}(1,\beta)=-\frac{p-1}{\beta\theta}W(1)v^{p-2}(1,\beta)\hat{w}'\left(-\frac{\log \beta}{\theta}\right).
\end{equation*}
Therefore, by using Proposition \ref{importantprop} again, we verify that there exists some $0<\beta_0<\beta^{*}$ such that $\dot{\lambda}(\beta)>0$ in $(0,\beta_0)$, $\dot{\lambda}(\beta_0)=0$, and $\dot{\lambda}(\beta_0)<0$ in $(\beta_0,\beta^{*})$. Moreover, we deduce that for any fixed $r\in [0,1]$, $v(r,\beta)$ is non-decreasing with respect to $\beta$ in $(0,\beta_0]$. 

Thus, it suffices to prove $\dot{\alpha}(\beta)>0$ for all $0<\beta<\beta^{*}$. We prove the assertion by contradiction. Thus, we assume that $\dot{\alpha}(\beta_1)\le 0$ for some $\beta_1\in (0,\beta^{*})$. 
Since $\alpha=\lambda^{-\frac{1}{p-1}}\beta-1$, we get $0<\beta_1<\beta_0$. By differentiating \eqref{vp} with respect to $\beta$, we have
\begin{equation*}
\left\{
\begin{alignedat}{4}
 -\Delta \dot{v}&=pV_k(|x|)v^{p-1}\dot{v}\hspace{4mm} \text{in } B_1,\\
\dot{v}(0)&=1, \hspace{4mm} \dot{v}(1)=\frac{1}{p-1}\lambda^{\frac{2-p}{p-1}}\dot{\lambda},\hspace{4mm}\dot{v}\in C^{0}[0,1]\cap C^{2}_{\mathrm{loc}}(0,1]\cap H^1(B_1).
\end{alignedat}
\right.
\end{equation*}
We first fix $\beta=\beta_0$. Then, since $\dot{\lambda}(\beta_0)=0$ and the fact that $v(r,\beta)$ is non-decreasing with respect to $\beta$ in $(0,\beta_0)$, we verify that $\dot{v}(r,\beta_0)$ is a positive eigenfunction for the operator $-\Delta_{D}-pV_{k} v^{p-1}$ corresponding to the eigenvalue $0$. Therefore, we deduce that $0$ is the first eigenvalue and thus $v(r,\beta_0)$ is stable. Since $v(r,\beta)$ is non-decreasing with respect to $\beta$, we verify that $v(r,\beta_1)$ is stable. Then, we denote  
$\alpha=\alpha(\beta_1)$, $\lambda=\lambda(\beta_1)$, and $u=u(r,\alpha(\beta_1))$. By differentiating \eqref{gelfand} with respect to $\beta$, we have
\begin{equation*}
\left\{
\begin{alignedat}{4}
 -\Delta \dot{u}&=\lambda pV_k(|x|)(1+u)^{p-1}\dot{u}+ \dot{\lambda}V_k(|x|)(1+u)^p\hspace{4mm} \text{in } B_1,\\
\dot{u}(0)&=\dot{\alpha}, \hspace{4mm} \dot{u}(1)=0,\hspace{4mm}\dot{u}\in C^{0}[0,1]\cap C^{2}_{\mathrm{loc}}(0,1]\cap H^1(B_1).
\end{alignedat}
\right.
\end{equation*}
Hence, by a similar argument to that in the proof of Lemma \ref{apriorileme} if $\dot{\alpha}=0$, we deduce that there exists $0<r_0\le 1$ such that $\dot{u}(r)<0$ in $(0,r_0)$ and $\dot{u}(r_0)=0$. We define 
\begin{equation*}
    \hat{u}(r)=
\begin{cases}
    \dot{u}(r) &\text{if $0\le r<r_0$,}\\
    0              &\text{otherwise.}
\end{cases}
\end{equation*}
Since $\dot{\lambda}>0$, we have 
\begin{align*}
    \int_{B_{1}}|\nabla \hat{u}|^2\,dx&= \int_{B_{1}}\lambda pV_k(|x|)(1+u)^{p-1}\hat{u}^2\,dx+ \dot{\lambda}\int_{B_{1}}V_k(|x|)(1+u)^{p}\hat{u}\,dx\\
    &<\int_{B_{1}}\lambda p V_k(|x|)(1+u)^{p-1}\hat{u}^2\,dx.
\end{align*}
On the other hand, since $v(r,\beta_1)$ is stable, we have 
\begin{align*}
\int_{B_{1}}\lambda p V_k(|x|)(1+u)^{p-1}\hat{u}^2\,dx&=\int_{B_{1}}p V_k(|x|){v(r,\beta_1)}^{p-1}\hat{u}^2\,dx\\
&\le \int_{B_{1}}|\nabla \hat{u}|^2\,dx, 
\end{align*}
which is a contradiction.
\end{proof}

\begin{proof}[Proof of Theorem \ref{singularth} \rm{(ii)} \textit{and Theorem \ref{mainthp}} \rm{(ii)}]
Assume that $p>p_c$. We first prove the uniqueness of a singular solution.
Let $(\lambda_{*}, U_{*})$ be a radial singular solution of \eqref{gelfand} for $f(u)=(1+u)^p$. 
Then, $v:=\lambda_{*}^{\frac{1}{p-1}}(U_{*}-1)$ is a singular solution of \eqref{apriorieqp} satisfying $v(1)=\lambda_{*}^{\frac{1}{p-1}}$. Moreover, we define $w(t):=v(r)W^{-1}(r)-1$ with $t=\log(-\log (r/e))$. Thanks to Lemma \ref{apriorilemmap}, Lemma \ref{lowerlemp}, and Theorem \ref{diagramthp}, we have $w>-1$ in $t\in \mathbb{R}$ and 
\begin{equation*}
-1<\limsup_{t\to\infty} w(t)\le C,
\end{equation*}
where $C>0$ is depending only on $k$ and $p$. Therefore, by \cite[Lemma 3.2]{Mi2020}, we get $w(t)\to 0$ as $t\to\infty$. Moreover, since the real parts of the associated eigenvalues of \eqref{lineap} are positive, we get $w(t)=0$. In particular, since $v(1)=W(1)$, it follows that $(\lambda_{*},U_{*})=(\theta(1-\theta), (\theta(1-\theta))^{-\frac{1}{p-1}}W+1)$ with $\theta=\frac{k}{p-1}$. Moreover, we can confirm that $U_{*}:=((\theta(1-\theta))^{-\frac{1}{p-1}}W+1)\in H^1_{0}(B_1)$ if and only if $p>p_c$.

Next, we show that the bifurcation curve converges to the singular solution. Let $(\lambda,u)=(\lambda(\beta), u(r,\alpha(\beta))$ be a radial solution of \eqref{gelfand} for $f(u)=(1+u)^p$. We define
$v:=\lambda^{\frac{1}{p-1}}(u-1)$ and $w(t):=v(r)W^{-1}(r)-1$ with $t=\log(-\log (r/e))$. In addition, we define $\hat{w}(s):=w(t)=w(t,\beta)$ with $s=t-\frac{\log \beta}{\theta}$. Then, we remind that $\hat{w}$ is the unique solution of \eqref{pw} with the initial condition \eqref{initialp}. Moreover, we remind that $W^{-1}(r)v(r,\beta)=\hat{w}(t-\frac{\log \beta}{\theta})+1$. 
Hence, thanks to Proposition \ref{importantprop},
for each $\varepsilon>0$, we obtain
$|v(r,\beta)|, |v'(r,\beta)|<C(\varepsilon)$
for all $r\in[\varepsilon, 2]$ and $\beta>1$, where $C(\varepsilon)>0$ is depending only on $\varepsilon$, $k$ and $p$. In addition, for each $\varepsilon>0$, there exists some $\beta(\varepsilon)$ depending only on $\varepsilon$, $k$ and $p$ such that $c<v(r,\beta)W^{-1}(r)<C$
for all $r\in[\varepsilon, 2]$ and $\beta>\beta(\varepsilon)$, where $0<c<C$ is depending only on $k$ and $p$. Therefore, by the elliptic regularity theory (see \cite{gil}), Arzel\`a-Ascoli theorem, and a diagonal argument, there exist a sequence $\{\beta_n\}_{n\in \mathbb{N}}$ and a singular solution $V\in C^2(0,1]$ of \eqref{vp} such that $\beta_n\to\infty$ and $v(r,\beta_n)\to V$ in $C^2_{\mathrm{loc}}(0,e]$ as $n\to\infty$. By the uniqueness of a singular solution, it follows that $V=W$ and thus we get the result.

Finally, we prove the oscillation of the bifurcation curve. Thanks to the above argument, we have $\hat{w}(s)\to 0$ as $s\to-\infty$. Since
\begin{equation*}
\lambda(\beta)=v^{p-1}(1,\beta)=\theta(1-\theta)(\hat{w}(-\theta^{-1}\log \beta)+1)^{p-1},
\end{equation*}
by Proposition \ref{oscup}, we deduce that $\lambda(\beta)$ turns around $\theta(1-\theta)$ infinitely many times when $p_c<p<p^{+}_{\mathrm{JL}}$. 
\end{proof}

\subsection{Stability of singular solutions}
We first prove Theorem \ref{morseth} (ii).
\begin{proof}[Proof of Theorem \ref{morseth} 
 \rm{(ii)}] For $\varphi\in C^{0,1}_{0}(B_1)$, we get
    \begin{align*}
   \mathcal{Q}_{U_{*}}(\varphi):=&\int_{B_1}|\nabla \varphi|^2\,dx-\int_{B_1} p\lambda_{*}V_k(|x|)(1+U_{*})^{p-1}\varphi^2\,dx\\
   &=\int_{B_1}|\nabla \varphi|^2\,dx-\int_{B_1}pV_k(|x|)W^{p-1}\varphi^2\,dx\\
   &=\int_{B_1}|\nabla \varphi|^2\,dx-\frac{kp}{p-1}\left(1-\frac{k}{p-1}\right)\int_{B_1}\frac{1}{|x|^2(\log(e/|x|))^2}\varphi^2\,dx.
\end{align*}
Since
\begin{equation*}
    \frac{kp}{p-1}\left(1-\frac{k}{p-1}\right)>\frac{1}{4} \hspace{4mm}\iff \hspace{4mm}p^{-}_{\mathrm{JL}}<p<p^{+}_{\mathrm{JL}},
\end{equation*}
thanks to Proposition \ref{critical prop}, we get the result.
\end{proof}
\begin{remark}
\label{strp}
\rm{By a similar method to that in the proof of Theorem \ref{morseth} (ii), we deduce that the singular solution $W$ is stable in $B_e$ if and only if $p_s<p\le p^{-}_{\mathrm{JL}}$ or $p^{+}_{\mathrm{JL}}\le p$. Moreover, by Theorem \ref{diagramthp} and Theorem \ref{singularth} (ii), we deduce that when $p_c<p$, we have $W\in H^1(B_1)$ and $v(r,\beta)>0$ for any $r\in(0,1]$, $\beta>0$.}
\end{remark}
Thanks to Remark \ref{strp}, we obtain the following separation result by a similar argument to that in the proof of Proposition \ref{sepae}.

\begin{proposition}
\label{sepap}
We assume that $p_{\mathrm{JL}}^{+}\le p$. Let $0<\beta<\gamma$ and $v(r,\beta)$ be the solution of \eqref{vp}. Then,
\begin{itemize}
    \item[{\rm{(i)}}] $v(r,\beta)<W(r)$ in $(0,e)$.
    \item [\rm{{(ii)}}] $v(r,\beta)<v(r,\gamma)$ in $(0,e)$. 
\end{itemize}
\end{proposition}
\begin{proof}[Proof of Theorem \ref{mainthp} \rm{(iii)}]
Let us denote by $\cdot$ the differentiation with respect to $\beta$. By Proposition \ref{sepap}, we have $\dot{\lambda}\ge 0$ for all $\beta\in \mathbb{R}$. Thus, it suffices to prove $\dot{\alpha}(\beta)>0$ for all $\beta\in \mathbb{R}$. We prove the assertion by contradiction. Thus, we assume that $\dot{\alpha}(\beta_0)\le 0$ for some $\beta_0$. Since it satisfies $(\alpha+1)\lambda^{\frac{1}{p-1}}=\beta$, we get $\dot{\lambda}(\beta_0)>0$.
By differentiating \eqref{gelfand} with respect to $\beta$, we have
\begin{equation*}
\left\{
\begin{alignedat}{4}
 -\Delta \dot{u}&=\lambda pV_k(|x|)(1+u)^{p-1}\dot{u}+ \dot{\lambda}V_k(|x|)(1+u)^p\hspace{4mm} \text{in } B_1,\\
\dot{u}(0)&=\dot{\alpha}, \hspace{4mm} \dot{u}(1)=0,\hspace{4mm}\dot{u}\in C^{0}[0,1]\cap C^{2}_{\mathrm{loc}}(0,1]\cap H^1(B_1).
\end{alignedat}
\right.
\end{equation*}
Hence, by using a similar argument to that in the proof of Lemma \ref{apriorileme} provided $\dot{\alpha}(\beta_0)=0$, we deduce that there exists $0<r_0\le 1$ such that $\dot{u}(r,\beta_0)<0$ in $(0,r_0)$ and $\dot{u}(r_0,\beta_0)=0$. We define 
\begin{equation*}
    \hat{u}(r)=
\begin{cases}
    \dot{u}(r,\beta_0) &\text{if $0\le r<r_0$,}\\
    0              &\text{otherwise.}
\end{cases}
\end{equation*}
Since $\dot{\lambda}(\beta_0)>0$, we have 
\begin{align*}
    \int_{B_{1}}|\nabla \hat{u}|^2\,dx&= \int_{B_{1}}\lambda(\beta_0) pV_k(|x|)(1+u)^{p-1}\hat{u}^2+\dot{\lambda}(\beta_0)V_k(|x|)(1+u)^{p}\hat{u}\,dx\\
    &<\int_{B_{1}}\lambda(\beta_0) pV_k(|x|)(1+u)^{p-1}\hat{u}^2\,dx.
\end{align*}
On the other hand, by Proposition \ref{sepap}, we get
\begin{align*}
\int_{B_{1}}\lambda(\beta_0) pV_k(|x|)(1+u)^{p-1}\hat{u}^2\,dx&= \int_{B_{1}} pV_k(|x|)v^{p-1}\hat{u}^2\,dx\\
&\le \int_{B_{1}} pV_k(|x|)W^{p-1}\hat{u}^2\,dx\le \int_{B_{1}}|\nabla \hat{u}|^2\,dx, 
\end{align*}
which is a contradiction.
\end{proof}

Finally, we prove Theorem \ref{sepathm} (ii) and (iii).
\begin{proof}[Proof of Theorem \ref{sepathm} \rm{(ii)} \textit{and} \rm{(iii)}]
We first remark that thanks to Proposition \ref{importantprop}, we get (ii). Moreover,  the result (iii) follows from Proposition \ref{sepap} (ii) in the case $p_{\mathrm{JL}}\le p$.
Then, it remains the case $p_c<p<p_{\mathrm{JL}}$. 
For $0<\beta<\gamma$, we define 
$w(t,\beta)=v(r,\beta)W^{-1}(r)-1$ with $t=\log (-\log(r/e))$. Here, we remark that $w(t,\beta)>-1$ for all $\beta>0$ and 
$w(t,\beta)\to 0$ as $t\to -\infty$ for all $\beta>0$ by Proposition \ref{importantprop} and the proof of Theorem \ref{singularth} (ii). In addition, we define $w_0(t)=w(t,\gamma)-w(t,\beta)$. Then
$w_0$ satisfies
\begin{equation*}
    \frac{d^2}{dt^2}w_0 + (2\theta-1)\frac{d}{dt}w_0+\theta(1-\theta) g(t)w_0, \hspace{4mm} t\in \mathbb{R}
\end{equation*}
with
\begin{equation*}
    g(t)=\frac{(w(t,\gamma)+1)^p-(w(t,\beta)+1)^p}{w(t,\gamma)-w(t,\beta)}-1\in C^0(\mathbb{R}).
\end{equation*}
Then, by the fact that $w(t,\beta)\to 0$ as $t\to -\infty$ for all $\beta>0$, we obtain $\theta(1-\theta)g(t)\to k(1-\theta)$ as $t\to -\infty$. Moreover, we remind that all nontrivial solutions of \eqref{lineap} change sign infinitely many times provided 
$p_c<p<p_{\mathrm{JL}}$. Thus, by Lemma \ref{strumlem}, we get the result.
\end{proof}
\section{Appendix}
In this section, we first prove the following
\begin{lemma}
\label{nlem}
Let $N=2$, $k\le 0$, and $V_k$ is that in \eqref{pote}. Then, if $u\in C^2(0,1]$ is a non-negative radial function satisfying \eqref{gelfand} for some $\lambda>0$ and a positive function $f\in C^0[0,\infty)$. Then, $\lim_{r\to 0} u(r)=\infty$.
\end{lemma}
\begin{proof}    
By a similar argument to that in the proof of Lemma \ref{apriorileme}, we have $u'\le 0$ in $(0,1)$. Assume to the contrary that
$\lim_{r\to 0} u(r)\to \alpha$ for some $\alpha>0$. We fix $0<r_1<r_2<1$ and define $m:=\inf\{f(t) ; 0\le t\le \alpha\}$.
Since $u$ satisfies \eqref{gelfand} and $u'(r_1)\le 0$, it follows for any $r_1<r<r_2$ that
\begin{equation*}
ru'(r)=r_1 u'(r_1)-\int^{r}_{r_1}\lambda sV_k(s)f(u(s))\,ds\le-\int^{r}_{r_1}\lambda sV_k(s)f(u(s))\,ds.
\end{equation*}
Letting $r_1\to 0$, we have
\begin{align*}
u'(r)\le -\frac{1}{r}\int^{r}_{0}\lambda sV_k(s)f(u(s))\,ds &\le  -\frac{m\lambda}{r} \int^{r}_{0} \frac{1}{s(-\log(s/e))^{2+k}}\,ds\\
&\le -\frac{m\lambda}{r} \int^{r}_{0} \frac{1}{s(-\log(s/e))^{2}}\,ds\\
&=\frac{m\lambda}{r\log(r/e)}.
\end{align*}
Integrating the above over $(\rho, r_2)$, we get 
\begin{equation*}
u(r_2)-u(\rho)\le m\lambda \log(-\log(r_2/e))- m\lambda \log (-\log(\rho/e)).
\end{equation*}
Thus, by letting $\rho\to 0$, we get a contradiction.
\end{proof}
Then, we prove the following type of Sturm's comparison theorem. 
\begin{lemma}
\label{strumlem}
Let $q\in \mathbb{R}$. We consider the following equations
\begin{equation*}
y''+ qy' + a(t)y = 0, \hspace{4mm}t\in \mathbb{R},
\end{equation*}
\begin{equation*}
z''+qz'+b(t)z = 0, \hspace{4mm} t\in \mathbb{R}.
\end{equation*}
We assume that there exists $l\in \mathbb{R}$ such that $a(t)>b(t)$ in $(-\infty,l)$. We assume in addition that a solution $z$ changes sign infinitely many times on $(-\infty,l)$. Then, any nontrivial solutions $y$ change sign infinitely many times on $(-\infty,l)$.
\end{lemma}
\begin{proof}
We assume that there exist $t_1<t_2<l$ such that $z(t_1)=z(t_2)=0$, $z(t)>0$ in $(t_1,t_2)$ and $y$ remains positive/negative in $[t_1,t_2]$.
Without loss of generality, we suppose that $y>0$ in $[t_1,t_2]$. Since 
\begin{equation*}
    [e^{qt}(z'y-y'z)]'=e^{qt}y(t)z(t)(a(t)-b(t))>0,
\end{equation*}
by integrating the above over $(t_1,t_2)$, we get
\begin{equation*}
    e^{qt_2}z'(t_2)y(t_2)- e^{qt_1}z'(t_1)y(t_1)>0.
\end{equation*}
Since $z'(t_2)\le 0$ and $z'(t_1)\ge 0$, we get a contradiction.
\end{proof}

\section*{Acknowledgments}
The author would like to thank Professor Michiaki Onodera for his valuable advice to clarify the presentation.
\bibliographystyle{plain}
\bibliography{bifurcation_2_weight}

\end{document}